\documentclass[12pt]{amsart}

\usepackage{latexsym,amsmath,amsfonts,amscd,amsthm,upref,graphics,amssymb}

\newtheorem{theorem}{Theorem}[section] % 1st argument is your name for it
\newtheorem{lemma}[theorem]{Lemma}     % 2nd argument is what is printed
\newtheorem{corollary}[theorem]{Corollary}
\newtheorem{proposition}[theorem]{Proposition}
\begin{document}
\title[Entropy of AT$(n)$ systems]{Entropy of AT$(n)$ systems}

\author[Radu-B. Munteanu]{Radu-B. Munteanu}
\address{Department of Mathematics, University of Bucharest\\
14 Academiei St, 010014, Sector 1\\
 Bucharest, Romania}
\email{radu-bogdan.munteanu@g.unibuc.ro}

\date{}

\begin{abstract}
In this paper we show that any ergodic measure preserving transformation of a standard probability space which is AT$(n)$ for some positive integer $n$ has zero entropy. We show that for every positive integer $n$ any Bernoulli shift is not AT($n$). We also give an example of a transformation which has zero entropy but does not have property AT($n$), for any integer $n\geq 1$.
\end{abstract}
\maketitle
%\tableofcontents

\section{Introduction}

In order to answer some questions from ergodic theory closely related to the theory of von Neumann algebras, T. Giordano and D. Handelman \cite{GH} reformulated matrix valued random walks and their associated group actions in terms of dimension spaces. Their approach leads to a notion of rank called AT$(n)$, for integers $n\geq 1$. This new concept generalizes approximate transitivity (shortly AT), a property of ergodic actions introduced by A. Connes and E. J. Woods \cite{CW} in the theory of von Neumann algebras, which occurs for $n=1$.

Throughout this paper $(X,\mathfrak{B},\mu,T)$ denotes a dynamical system, where $T$ is a measure preserving automorphism of a standard probability space $(X,\mathfrak{B},\mu)$. For an integer $n\geq 1$, we say that the dynamical system $(X,\mathfrak{B},\mu,T)$ (or simply $T$) is AT$(n)$ if for any $\varepsilon>0$, for any finite set of functions $\{f_{i}\}_{i=1}^{k}$ from $ L^{1}_{+}(X,\mu)$ there exist $n$ functions $\{g_{m}\} _{m = 1, . . . , n}\in L^{1}_{+}(X,\mu)$, a
positive integer $N$, nonnegative reals $\{\alpha_{i,j}^{(m)}\}^{m=1,2,\ldots,n}_{i=1,2,\ldots,k,  j=1,2,...,N}$ and integers $\{t^{(m)}_j \}^{m=1,\ldots,n}_{j=1,\ldots,N}$, such that
\begin{equation}\label{ATn}
\|f_{i}-\sum_{m=1}^{n}\sum_{j=1}^{N}\alpha_{i,j}^{(m)}g_{m}\circ T^{t_{j}^{(m)}}\|_{1}<\varepsilon,
\end{equation}
for $i=1,2,\ldots, k$.

Note it is sufficient to ask that equation (\ref{ATn}) holds for $k=n+1$. Also, one can demand that $\|g_{m}\|=1$ for all $m$ and that $\sum_{m=1}^{n}\sum_{j=1}^{N}\alpha_{i,j}^{(m)}=\|f_{i}\|$, for all $i$.

Remark that a rank $n$ transformation (see \cite{F}) is AT$(n)$ and every AT$(n)$ system enjoys AT$(n+1)$ property. The techniques developed by T. Giordano and D. Handelman  in \cite{GH} allowed the authors to construct examples of AT$(2)$ transformation which are not AT. In \cite{L}, it was proved that the measure preserving automorphism corresponding to the Rudin-Shapiro substitution, which has rank $4$ (and therefore is AT$(4)$), is not AT.

Dynamical entropy is an invariant of measure theoretic dynamical systems introduced by A. N. Kolmogorov \cite{K} and brought to its contemporary form by Y. G. Sinai  \cite{Si}.  In 1970, D. Ornstein \cite{O} showed that Kolmogorov-Sinai entropy completely classifies the Bernoulli shifts, a basic problem which couldn't be solved for many decades. It was proved by A. Connes and E. J. Woods \cite{CW} that any dynamical system which is AT has zero entropy. Different proofs of this result can be found in \cite{D}, \cite{DQ} and \cite{L}. It is natural to ask whether AT$(n)$ dynamical systems have also zero entropy.

In this paper we give a necessary condition for shift maps to be AT$(n)$ (Theorem 2.1). For such transformations, this is a generalization of the necessary condition for an action to be AT from \cite{DQ}.

We use this condition to prove that, for any positive integer $n$, Bernoulli shifts are not AT$(n)$, for $n\geq 1$. An important consequence of this result is Corollary 3.3, which shows that any finite measure preserving transformation which is AT$(n)$ for some positive integer $n$, has zero entropy. In \cite{DQ} A. Dooley and A. Quas proved that zero entropy is not sufficient for AT; they gavr an example of a zero entropy transformation which is not approximately transitive. 
In this paper we prove that the zero entropy and not AT transformtion from \cite{DQ} is not AT($n$), for any $n\geq 1$.

\section{A necessary condition for shift maps to be AT$(n)$}

\noindent Let $\Sigma_{k}=\{1,2,\ldots,k\}^{\mathbb{Z}}$ be the shift space over the alphabet $\{1,2,\ldots,k\}$. A cylinder set in $\Sigma_{k}$ is a set of the form $[y_{0},y_{1}, \ldots , y_{m}]_{n}=\{x\in \Sigma_{k}; x_{n}=y_{0}, x_{n+1}=y_{1},\ldots ,x_{m+n}=y_{m}\}$, where $m,n\in\mathbb{Z}$, $m\geq 0$ and $y_{i}\in\{1,2,\ldots k\}$. The cylinders of the form $[y_{0},y_{1}, \ldots, y_{2n}]_{-n}$ with $n\geq 0$ are called centered cylinders. We denote by $\mathfrak{B}_{k}$ the $\sigma$-algebra generated by the cylinder sets of the shift space $\Sigma_{k}$. The map  $S_k:\Sigma_{k}\rightarrow\Sigma_{k}$ defined by $$(S_kx)_{n}=x_{n+1} \text{ for }x=(x_{n})_{n\in \mathbb{Z}}$$
is called the $k$-shift map.

Let $p = (p(1),\ldots, p(k))$ be a probability vector with non-zero entries, i.e. $p(i)>0$ for $i=1,2,\ldots, k$ and $\sum_{i=1}^{k}p(i)=1$.
%Let $\mu_{p}$ be the Bernoulli measure determined by $p$. For cylinder sets $\mu_{p}$ is given by
%\[\mu_{p}([y_{0},y_{1}, \ldots y_{m}]_{n})=p(y_{0})p(y_{1})\cdots p(y_{m}).\]
Let $\mu_{p}$ be the unique probability measure on $(\Sigma_{k},\mathfrak{B}_{k})$, which on cylinder sets is given by
\[\mu_{p}([y_{0},y_{1}, \ldots y_{m}]_{n})=p(y_{0})p(y_{1})\cdots p(y_{m}).\]
This measure is called the Bernoulli measure determined by $p$. The dynamical system $(\Sigma_{k},\mathfrak{B}_{k},\mu_{p},S_k)$ is called the Bernoulli shift associated to the probability vector $p$.

Let $\Lambda$ be a finite subset of the integers. A funny word on the alphabet $\{1,2,\ldots,k\}$ based on $\Lambda$ is a finite sequence $W=(W_{n})_{n\in \Lambda}$ with $W_{n}\in \{1,2,\ldots,k\}$. For two funny words $W,W'$
based on the same set $\Lambda$ their Hamming distance is given by
\[d_{\Lambda}(W,W')=\frac{1}{|\Lambda|}\text{card}\{n\in\Lambda: W_{n}\neq W'_{n}\}.\]
If $x\in\Sigma_{k}$ and $\Lambda$ is a finite set in $\mathbb{Z}$ we denote by $x|_{\Lambda}$ the funny word $(x_{n})_{n\in\Lambda}$.

The following theorem provides a necessary condition for shift maps to be AT$(n)$.

\begin{theorem}\label{t1}
Let $S_k$ be the shift map on the space $\Sigma_{k}$ and $\nu$ be a non atomic shift invariant probability measure on $\Sigma_{k}$. Assume that $S_k$ is AT$(n)$ but not AT$(n-1)$, for some $n\geq 2$.
%Let $(\varepsilon_{m})_{n\geq 1}, (\delta_{m})_{n\geq 1}$ be two sequence of positive reals converging to zero.
Then for every $\varepsilon>0$ and every  $\delta>0$ there exist finite sets $\Lambda^{1},\Lambda^{2},\ldots,\Lambda^{n}$, with
$\min\{|\Lambda^{i}|; i=1,2,\ldots,n\}$ arbitrarily large, and funny words $W^{i}$ based on $\Lambda^{i}$ for $i=1,2,\ldots,n$ such that
%\[|\Lambda^{1}|\nu(B_{\varepsilon,x^{1},\Lambda^{1}})+|\Lambda^{2}|\nu(B_{\varepsilon,x^{2},\Lambda^{2}})+\cdots+|\Lambda^{n}|\nu(B_{\varepsilon,x^{n},\Lambda^{n}})>1-\varepsilon.\]
\[\sum_{i=1}^{n}|\Lambda^{i}|\nu\left(\{x\in\Sigma_{k}: d_{\Lambda^{i}}(x|_{\Lambda},W^{i})<\varepsilon\}\right)>1-\delta.\]
\end{theorem}

\begin{proof}
Let $(\delta_{m})_{m\geq 1}$ be a sequence decreasing to zero. The theorem will result if we prove that for every $\varepsilon>0$ there exist finite sets
$\Lambda^{i}_{m}\subset \mathbb{Z}$ and funny words $W^{i}_{m}$ based on $\Lambda^{i}_{m}$ for $i=1,2, \ldots n$ such that
$\sup_{m\geq 1}\min\{|\Lambda^{i}_{m}|,i=1,2,\ldots, n\}=\infty$ and
\[\sum_{i=1}^{n}|\Lambda^{i}_{m}|\nu\left(\{x\in\Sigma_{k}: d_{\Lambda^{i}}(x|_{\Lambda},W^{i}_{m})<\varepsilon\}\right)>1-\delta_{m}.\]

Let $\varepsilon>0$. For $m\geq 0$, denote by $\mathcal{C}_{m}$ the set of all centered cylinders of the form $[y_{0},y_{1}, \ldots, y_{2m}]_{-m}$ which have positive measure. Since $\nu$ is non atomic, for each $C\in\mathcal{C}_{m}$, one can find a measurable partition $\mathcal{P}_{C}$ of $C$ such that $\nu(A)<\delta_{m}$, for every $A\in \mathcal{P}_{C}$. Let $$\mathcal{J}_{m}=\{ A\in \mathcal{P}_{C}: \ C\in \mathcal{C}_{m}\}.$$
Notice that $\{A: \text{ there exists }m\geq 1\text{ such that }A\in\mathcal{J}_{m}\}$ generates (up to null sets) the sigma algebra $\mathfrak{B}_{k}$.

For $A\in \mathcal{J}_{m}$ let $g_{A}=\frac{1}{\nu(A)}{\bf 1}_{A}$ the normalized indicator function corresponding to $A$. Let $m\geq 1$ and $A\in\mathcal{J}_{m}$. Since, by assumption, $S_k$ is AT$(n)$, there exists $f_{1,m}, f_{2,m},\ldots, f_{n,m}\in L^{1}_{+}(\Sigma_{k},\nu)$ of norm $1$, sequences of non-negative numbers $\{a^{1}_{A,j}\}_{j\in\mathbb{Z}},\{a^{2}_{A,j}\}_{j\in\mathbb{Z}},\ldots,$ $\{a^{n}_{A,j}\}_{j\in\mathbb{Z}}$ with finitely many non-zero elements such that $\sum_{j}a^{1}_{A,j}+\sum_{j}a^{2}_{A,j}+\cdots + \sum_{j}a^{n}_{A,j} =1$ and
\[\|g_{A}-\sum_{i=1}^{n}\sum_{j}a^{i}_{A,j}f_{i,m}\circ S_k^{-j}\|<\frac{\varepsilon\delta_{m}^{2}}{36}.\]
It then follows that for any $A \in\mathcal{J}_{m}$ we have
\[\int_{\Sigma_{k}\setminus A}\sum_{i=1}^{n}\sum_{j}a^{i}_{A,j}f_{i,m}\circ S_k^{-j} d\nu <\frac{\varepsilon\delta_{m}^{2}}{36} .\]
Let $A\in\mathcal{J}_{m}$. For $i=1,2,\ldots, n$ define
$$P_{A}^{i}=\{j: \int_{\Sigma_{k}\setminus A}f_{i,m}\circ S_k^{-j}d\nu\geq\varepsilon\delta_{m}/6\}.$$ It easily can be seen that $\sum_{j\in P^{1}_{A}}a^{1}_{A,j}+\sum_{j\in P^{2}_{A}}a^{2}_{A,j}+\cdots +\sum_{j\in P^{n}_{A}}a^{n}_{A,j}<\delta_{m}/6.$ By setting the $a^{i}_{A,j}$ to be $0$ for $j\in P^{i}_{A}$, and  re-scaling the remaining $a^{i}_{A,j}$ we obtain coefficients $b^{1}_{A,j}, b^{2}_{A,j},\ldots ,b^{n}_{A,j}$ with $\sum_{j}b^{1}_{A,j}+\sum_{g}b^{2}_{A,j}+\sum_{j}b^{n}_{A,j}=1$, satisfying
\[\|g_{A}-\sum_{i=1}^{n}\sum_{j}b^{i}_{A,j}f_{i,m}\circ S_k^{-j}\|<\frac{\delta_{m}}{2},\]
and such that
\[\int_{\Sigma_{k}\setminus A} f_{i,m}\circ S_k^{-j}d\nu<\frac{\varepsilon\delta_{m}}{6}\]
if $b^{i}_{A,j}>0$.
For $i=1,2,\ldots, n$, let $$\Lambda^{i}_{m}=\{j\in\mathbb{Z}: \text{ there exists }A\in\mathcal{J}_{m}, b^{i}_{A,j}>0\}.$$
We claim that
\begin{equation}\label{cc}
\sup_{m\geq 1}\min\{|\Lambda^{i}_{m}|,i=1,2,\ldots, n\}=\infty.
\end{equation}
We will prove the claim by contradiction. Let us suppose that $$\min\{|\Lambda^{i}_{m}|,1\leq i\leq n\}\leq M<\infty,\text{ for all }m\geq 1.$$
Let $i(m)$ be such that $|\Lambda^{i(m)}_{m}|=\min\{|\Lambda^{i}_{m}|,1\leq i\leq n\}$, then  $|\Lambda^{i(m)}_{m}|\leq M$, for all $m\geq 1$.
Let $$\mathcal{L}_{m}=\{A\in\mathcal{J}_{m}: \text{ there exists }j\in\mathbb{Z}\text{ such that }b^{i(m)}_{A,j}>0\}.$$Remark that $|\mathcal{L}_{m}|\leq|\Lambda^{i(m)}_{m}|$. Since $\nu(A)<\delta_{m}$ for every $A\in \mathcal{J}_{m}$, it follows that $\lim_{m\rightarrow\infty}\nu(\cup_{A\in \mathcal{L}_{m}}A)=0$. Hence, any $f\in L^{1}_{+}(\Sigma_{k},\nu)$ can be approximated arbitrarily close in $L^{1}$-norm by step functions of the form $\sum_{A\in\mathcal{J}_{m}-\mathcal{L}_{m}}\alpha_{A}g_{A}$ with $\alpha_{A}\geq 0$ and  $\sum_{A\in\mathcal{J}_{m}\setminus\mathcal{L}_{m}}\alpha_{A}=\|f\|$, by choosing $m$ sufficiently large. 
 
Let $f_{l}\in L^{1}_{+}(\Sigma_{k},\nu)$, $l=1,2,\ldots, n$ be functions of norm $1$ and $\eta>0$. Choose $m\geq 1$
sufficiently large such that $\delta_{m}<\eta$ and such that exist non-negative numbers $\alpha_{A}^{i}$  with $\sum_{A\in\mathcal{J}_{m}\setminus\mathcal{L}_m} \alpha^{l}_{A}=1$ satisfying
\[\|f_{l}-\sum_{A\in\mathcal{J}_{m}\setminus\mathcal{L}_m}\alpha^{l}_{A} g_{A}\|_{1} <\frac{\eta}{2}, \]for $l=1,2,\ldots,n$.
Note that if $A\in \mathcal{J}_{m}\setminus\mathcal{L}_{m}$ then
\[\|g_{A}-\sum_{i=1, i\neq i(m)}^{n}\sum_{j}b^{i}_{A,j}f_{i,m}\circ S_k^{-j}\|<\frac{\delta_{m}}{2}.\]
We obtain then non-negative coefficients $c^{i,l}_{j}$, $i,l\in\{1,2,\ldots n\}$, $i\neq i(m)$ with finitely many of them different from zero
such that $$\sum _{i=1, i\neq i(m)}^{n}\sum_{j}c^{i,l}_{j}=1$$ and
\[\|f_{l}-\sum_{i=1, i\neq i(m)}^{n}\sum_{j}c^{i,l}_{j} f_{i,m}\circ S_k^{-j}\| <\eta.\]
Since such an approximation can be done for any $\eta>0$, it follows that $S_k$ is AT$(n-1)$. But this contradicts our hypothesis, and therefore, the claim (\ref{cc}) holds.

Notice that for $i=1,2,\ldots, n$ we have
\[\sum_{A\in\mathcal{J}_{m}}\sum_{j: b^{i}_{A,j}>0}\int_{\Sigma_{k}\setminus A}f_{i,m}\circ S_k^{-j}d\nu<\frac{\varepsilon\delta_{m}|\Lambda^{i}_{m}|}{6}\]
and then,
\begin{align*}
\frac{\varepsilon\delta_{m}|\Lambda_{m}^{i}|}{6}&>\sum_{A\in\mathcal{J}_{m}}\sum_{j: b^{i}_{A,j}>0}\int f_{i,m}\circ S_k^{-j}\cdot1_{\Sigma_{k}\setminus A}d\nu=\int f_{i,m}\sum_{A\in\mathcal{J}_{m}}\sum_{g: b^{i}_{A,j}>0} 1_{\Sigma_{k}\setminus A}\circ S_k^{j}d\nu.
\end{align*}
Let \[H^{i}_{m}=\frac{1}{|\Lambda_{m}^{i}|}\sum_{A\in\mathcal{J}_{m}}\sum_{j: b^{i}_{A,j}>0}1_{\Sigma_{k}\setminus A}\circ S_k^{j}.\]
For $j\in\Lambda_{m}^{i}$, there exists a unique set $A\in\mathcal{J}_{m}$ such that $\int_{A} f_{i,m}\circ S_k^{-j}d\nu>1-\varepsilon\delta_{m}/6$. Let $[z]_{0}=\{x\in\Sigma_{k}: x_{0}=z\}$ be the unique cylinder set from $\mathcal{C}_{1}$ containing $A$ and define $W_{m,j}^{i}$ to be this $z$. Denote by $W^{i}_{m}$ the funny word $(W_{m,j}^{i})_{j\in\Lambda_{m}^{i}}$ based on $\Lambda_{m}^{i}$.
Since the above inequality demonstrates that $\int H^{i}_{m}f_{i,m} d\nu < \varepsilon\delta_{m}/6$, it follows that $\int_{\{z: H^{i}_{m}(z)>\varepsilon\}}f_{i,m} (x)d\nu < \delta_{m}/6$. Let $\widetilde{f}_{i,m}$ be the function defined by
\begin{equation*}
\widetilde{f}_{i,m}(x)=\left\{
\begin{array}{l}
0 \ \ \ \ \ \ \ \ \ \ \ \ \ \ \ \ \ \ \ \ \ \ \ \ \ \ \ \ \ \ \ \ \ \text{ if }   H^{i}_{m}(x)>\varepsilon \\[0.2cm]
f(x)/\int_{\{z: H^{i}_{m}(z)<\varepsilon\}}f_{i,m} d\nu  \ \ \ \ \ \text{ otherwise}.\\[0.2cm]

\end{array}
\right.
\end{equation*}
Clearly, \[H^{i}_{m}(x)\geq \frac{1}{\Lambda_{m}^{i}}\text{card}\{j\in \Lambda_{m}^{i}: W^{k}_{m,j}\neq x_{j}\}.\]
Therefore, the support of $\widetilde{f}_{i,m}$ is contained in $\{x\in\Sigma_{k}: d(x|_{\Lambda^{i}_{m}},W^{i}_{m})<\varepsilon\}.$
Since $\|\widetilde{f}_{i,m}-f_{i,m}\|<\frac{\delta_{m}}{2}$ for $i=1,2,\ldots, n$, for  $A\in\mathcal{J}_{m}$ we have that
\[\|g_{A}-\sum_{i=1}^{n}\sum_{j}b^{i}_{A,j}\widetilde{f}_{i,m}\circ S_k^{-j}\|<\delta_{m}.\]
Hence, summming over $A\in\mathcal{J}_{m}$ we get
\[\|1-\sum_{i=1}^{n}\sum_{A\in\mathcal{J}_{m}}\sum_{j}\nu(A)b^{i}_{A,j}\widetilde{f}_{i,m}\circ S_k^{-j}\|<\delta_{m}.\]
Since the support of each $\widetilde{f}_{i,m}$ is contained in $\{x\in\Sigma_{k}: d_{\Lambda^{i}_{m}}(x|_{\Lambda^{i}_{m}},W^{i}_{m})<\varepsilon\}$, it follows that
$$\sum_{i=1}^{n}\sum_{A\in\mathcal{J}_{m}}\sum_{j}\nu(A)b^{i}_{A,j}\widetilde{f}_{i,m}\circ S_k^{-j}$$
is supported on a set of measure at most \[\sum_{i=1}^{n}|\Lambda^{i}_{m}|\{x\in\Sigma_{k}: d_{\Lambda^{i}_{m}}(x|_{\Lambda^{i}_{m}},W^{i}_{m})<\varepsilon\}.\] Therefore
\[\sum_{i=1}^{n}|\Lambda^{i}_{m}|\nu\left(\{x\in\Sigma_{k}: d_{\Lambda^{i}_{m}}(x|_{\Lambda},W^{i}_{m})<\varepsilon\}\right)>1-\delta_{m}\]
and the proof of the theorem is complete.
\end{proof}

\section{AT$(n)$ systems have zero entropy}

In this section we show that for any positive integer $n$, Bernoulli shifts are not AT$(n)$. We also show that AT$(n)$ systems have zero entropy. Let us prove first the following lemma.
\begin{lemma}\label{xx}
A factor of an AT$(n)$ system is AT$(n)$.
\end{lemma}
\begin{proof}
Let $\pi$ be a factor map from an AT$(n)$ system $(X,\mathfrak{B},\mu,T)$ onto another dynamical system $(Y,\mathfrak{F},\nu,S_k)$. Let $f_{1},f_{2},\ldots, f_{n+1}\in L^{1}_{+}(X,\mu)$. Since $(X,\mathfrak{B},\mu, T)$ is AT$(n)$, there exists $g_{1},g_{2},\ldots,g_{n}\in L^{1}_{+}(X,\mu)$,  a
positive integer $N$, reals $\alpha_{i,j}^{(m)}\geq 0$, for  $m=1,2,\ldots,n$, $i=1,2,\ldots,n+1$, $j=1,2,...,N$ and integers $\{t^{(m)}_j \}^{m=1,...,n}_{j=1,...,N}$, such that
\[\|f_{i}\circ\pi-\sum_{m=1}^{n}\sum_{j=1}^{N}\alpha_{i,j}^{(m)}g_{m}\circ T^{t_{j}^{(m)}}\|_{1}<\varepsilon,\]
for $i=1,2,\ldots , n + 1$. Taking expectation with respect to the $T$-invariant $\sigma$-algebra $\pi^{-1}(\mathfrak{B})$, we obtain
\[\|f_{i}\circ\pi-\sum_{m=1}^{n}\sum_{j=1}^{N}\alpha_{i,j}^{(m)}E(g_{m}|\pi^{-1}(\mathfrak{B}))\circ T^{t_{j}^{(m)}}\|_{1}<\varepsilon,\]
for each $i$. Notice that for all $i$, we can write $E(g_{m}|\pi^{-1}(\mathfrak{B})=G_{m}\circ\pi$ for some measurable function $G_m$ on $Y$, and then, since $S_k\circ\pi=\pi\circ T$, we have
\[\|f_{i}-\sum_{m=1}^{n}\sum_{j=1}^{N}\alpha_{i,j}^{(m)}G_{m}\circ S_k^{t_{j}^{(m)}}\|_{1}<\varepsilon.\]
We can then conclude that the system  $(Y,\mathfrak{F},\nu,S_k)$ is AT$(n)$.
\end{proof}
\begin{proposition}
Let $(\Sigma_{k},\mathfrak{B}_{k},\mu_{p},S_k)$ be the Bernoulli shift associated to the probability vector $p=(p(1),\ldots p(k))$. Then, for any $n\geq 1$, the shift map $S_k$ is not AT$(n)$.
\end{proposition}

\begin{proof}
We prove the proposition by induction. It is well known that $S_k$ has positive entropy and therefore is not AT$(1)$.
Let us assume now that for some $n\geq 2$, $S_k$ is not AT$(n-1)$.
%We will show that $S_k$ is not AT$(m)$.
Let $r=\max\{p(i); i=1,2,\ldots,k\}$.
If $W$ is a funny word based on $\Lambda$ then \[\mu_{p}(x\in\Sigma_{k}: d_{\Lambda}(x|_{W},W))\leq\left(\begin{array}{c}
    m \\
    \left[m\varepsilon\right]
  \end{array}
\right)r^{m-[m\varepsilon]},\]where $|\Lambda|=m$.

Notice that if $\varepsilon$ is sufficiently small then
\[\frac{r(1-\varepsilon)^{\varepsilon}}{(1-\varepsilon)\varepsilon^{\varepsilon}r^{\varepsilon}}<1. \]
For $m$ sufficiently large, we have
\[\left( \begin{array}{c}
    m \\
    \left[m\varepsilon\right]
  \end{array}
\right)r^{m-[m\varepsilon]}<\frac{2(\frac {m}{e})^{m}\sqrt{2\pi m}}{(\frac{m\varepsilon}{e})^{m\varepsilon}\sqrt{2\pi m\varepsilon}
(\frac{m(1-\varepsilon)}{e})^{m(1-\varepsilon)}\sqrt{2\pi m(1-\varepsilon)}}r^{m-[m\varepsilon]}\]\[<\frac{1}{\sqrt{2\pi m\varepsilon(1-\varepsilon)}}
\left(\frac{r(1-\varepsilon)^{\varepsilon}}{(1-\varepsilon)\varepsilon^{\varepsilon}r^{\varepsilon}}\right)^{m}<\frac{1-\varepsilon}{n\cdot m}.\]
It then follows that if $|\Lambda|$ sufficiently large, $\mu_{p}(x\in\Sigma_{k}: d_{\Lambda}(x|_{W},W)<\varepsilon)<\frac{1-\varepsilon}{n\cdot|\Lambda|}$. Then Theorem \ref{t1}, implies that $S_k$ is not AT$(n)$.
\end{proof}

We can  prove now the result announced in the beginning concerning the entropy of AT$(n)$ systems.

\begin{corollary}
Let $n$ be a positive integer and let $T$ be an ergodic measure preserving transformation of a standard probability space $(X,\mathfrak{B},\mu)$ which is AT$(n)$. Then $T$ has zero entropy.
\end{corollary}

\begin{proof}
We prove this lemma by contradiction. Assume that the entropy $h(T)$ of $T$ is strictly positive. Consider a Bernoulli shift $(\Sigma_{k},\mathfrak{B}_{k},\mu_{p},S_k)$ associated to a probability vector $p=(p(1),p(2)$, $\ldots,p(k))$ such that $h(T)\geq h(S_k)=\sum_{i=1}^{k}p(i)\log p(i)$. By Sinai's theorem, $(\Sigma_{k},\mathfrak{B}_{k},\mu_{p},S_k)$ is a factor of the system $(X,\mathfrak{B},\mu,T)$. Then, Lemma \ref{xx}, implies that $(\Sigma_{k},\mathfrak{B}_{k},\mu_{p},S_k)$ is AT$(n)$. This is a contradiction. 
\end{proof}

\bigskip
In the last part of this setion we will show that there exists  a zero entropy dynamical system which is not AT($n$) for any $\geq 1$. 
Let $\alpha$ be an irrational number and let $T$ be the transformation of the 2-torus $\mathbb{T}^2$ defined by
$$T(s,t)=(s+\alpha, 2s+t+\alpha) \ \ (\text{mod }1).$$ This transformation, studied by H. Furstenberg \cite{Fu1,Fu2}, is measure preserving, uniquely ergodic and has zero entropy. It was proved by A. Dooley and A. Quas in \cite{DQ} that $T$ is not approximately transitive. 

\begin{proposition}
The zero entropy transformation $T$ defined above does not have property AT($n$), for any positive integer $n$.
\end{proposition}

\begin{proof}
For $k\geq 2$ let $\mathcal{P}$ be the partition of $\mathbb{T}^2$ consisting of the sets $A_{i}=\mathbb{T}\times \left[\frac{i-1}{k},\frac{i}{k}\right)$, for $i=1,2,\ldots, k+1$. Let $\Sigma_{k+1}=\{1,2,\ldots, k+1\}^{\mathbb{Z}}$ and let $\pi:\mathbb{T}^2\rightarrow \Sigma_{k+1}$ be the natutal map from $\mathbb{T}^2$ in $\Sigma_{k+1}$, defined by $\pi(z)=(x_n)_{n\in\mathbb{Z}}$, where $x_n=i$ if $T^{n}(z)\in A_i$. 

If $\mu$ is the Haar measure $\mu$ of $\mathbb{T}^2$, denote by $\nu$ the measure $\pi^{-1}\circ \mu$ induced by $\mu$ on $\Sigma_{k+1}$.

%Denote by $\nu$ the measure $\Sigma_{k+1}$ we consider the measure $\nu$ induced by the partition $\mathcal{P}$ from the Haar measure $\mu$ of $\mathbb{T}^2$. More precisely, $\nu(A)=\mu(\pi^{-1}(A))$ for every measurable subset $A$ of $\Sigma_{k+1}$.

\noindent We show first that for all $m\neq n$
\begin{equation}\label{eq1}
\nu\left(\{y\in \Sigma_{k+1}: \  y_m=i, \quad  y_n=j\}\right)=\frac{1}{(k+1)^2}
\end{equation} 
Since $\nu$ is shift invariant it is enough to show that for every $n\in\mathbb{Z}$ $\nu\left(\{y\in \Sigma_{k+1}: \  y_0=i, \quad y_n=j\}\right)=\frac{1}{(k+1)^2}$. This is the Haar measure of the set of points $(s,t)$ such that $s\in A_i$ and $\pi_2(T^{n}(s,t))\in A_j$ (here $\pi_2(s,t)=t$ for $(s,t)\in\mathbb{T}^2$). In other words, this is the set of all $(s,t)$ such that $s\in A_i$ and $\langle t+n^2\alpha+2ns\rangle\in A_j$, where $\langle z\rangle$ denotes the fractional part of $z\in\mathbb{R}$. It easily can be observed that the measure of this set is $\frac{1}{(k+1)^2}$.
Let $\Lambda$ be an arbitrary subset of integers of cardinality $n$ and fix $x\in \Sigma_{k+1}$. Define $\Lambda_{j}:\Sigma_{k+1}\rightarrow\mathbb{C}$ by

 %let \begin{equation*}
%\Lambda_{j}(z)=\left\{\begin{array}{ccc}
                             % 1 &  \text { if }  & y_j= x_j, \ \ \ \ \ \  \\
                              %\varepsilon &   \text{ if } & y_j= x_j+1,  \\
                              %\cdots\\
                              %\varepsilon^{k}&  \  \text{ if } & y_j=x_j+k,
%\end{array}
%\right.
%\end{equation*}

\begin{equation*}
\Lambda_{j}(y)=\begin{cases}
                              1 &  \text { if } \quad y_j= x_j, \\
                              \varepsilon &   \text{ if }\quad  y_j= x_j+1,  \\
                              \cdots\\
                              \varepsilon^{k} &   \text{ if }\quad  y_j=x_j+k, 
\end{cases}
\end{equation*}
where $\varepsilon=e^{2\pi i/(k+1)}$. Let $S=\sum_{j\in \Lambda} \Lambda_j$. From (\ref{eq1}) it results that $E(|S|^2)=n$ and then, by Markov inequality we get 
\begin{align*}\nu\left(|S|>\frac{2k+1}{2k+2}\cdot n\right)&= \ \nu\left(|S|^2>\frac{(2k+1)^2}{(2k+2)^2}\cdot n^2\right)\\& \
\leq \frac{E(|S|^2)}{n^2}\cdot\frac{(2k+2)^2}{(2k+1)^2}=\frac{(2k+2)^2}{(2k+1)^2}\cdot\frac{1}{n}.
\end{align*}
For $y\in \Sigma_{k+1}$ define $$a_{i}(y)=\text{card}\{{j\in \Lambda: \  y_j=x_j+i-1}\}\text{ for }i=1,2,\ldots k+1.$$ Then  $$S=a_1+\varepsilon a_2+\varepsilon^2 a_3+\cdots+\varepsilon^{k} a_{k+1}.$$ 
Now, consider the sets
\begin{align*}B_{\frac{1}{4k+4},x,\Lambda}&=\left\{y\in \Sigma_{k+1}: \  d_{\Lambda}(y|_{\Lambda},x|_{\Lambda})<\frac{1}{4k+4}\right\},\\A&=\left\{y\in \Sigma_{k+1}: \ |S|>\frac{2k+1}{2k+2}\cdot n\right\}\end{align*}
and for $i=1,2,\ldots ,k+1$,
\begin{align*}
A_i=\left\{y\in A : \ a_{i}(y)>a_j(y)\text{ for all }j\neq i\right\}.
\end{align*}
Notice that $B_{\frac{1}{4k+4},x,\Lambda}\subset A_1$. Indeed if $y\in B_{\frac{1}{4k+4},x,\Lambda}$, then $a_{1}(y)\geq \frac{4k+3}{4k+4}\cdot n$, and consequently,
\begin{align*}
|S(y)|>|a_1(y)|&-|\varepsilon a_{2}(y)+\cdots \varepsilon^{k}a_{k+1}(y)|\\ &\geq \left(\frac{4k+3}{4k+4}-\frac{1}{4k+4}\right)\cdot n=\frac{2k+1}{2k+2}\cdot n.
\end{align*}
Denote by $R$ be the transformation of $\mathbb{T}^2$ defined by $R(s,t)=(s,t+\frac{1}{k+1})$ $(\text{mod }1)$, which clearly commutes with $T$. It is easy to see that $a_{i}(\pi(s,t))=a_{i+1}(\pi(R(s,t)))$ for all $(s,t)\in \mathbb{T}^2$ and $i\in\{1,2,\ldots,k\}$ and then  $R(\pi^{-1}(A_i))=\pi^{-1}(A_{i+1})$ for every $i\in\{1,2,\ldots, k\}$. Since $R$ is a measure preserving transformation it follows that 
$\nu(A_i)=\mu(\pi^{-1}(A_i))=\mu(R(\pi^{-1}(A_i)))=\mu(\pi^{-1}(A_{i+1}))=\nu(A_{i+1})$ for every $i\in\{1,2,\ldots, k\}$ and then, $\nu(A_1)\leq \frac{1}{k+1}\cdot\nu(A)$. Hence
\[|\Lambda|\cdot\nu(B_{\frac{1}{4k+4},x,\Lambda})\leq |\Lambda|\cdot\nu(A_1)\leq \frac{4k+4}{4k^2+4k+1}\]and therefore
\begin{equation}\label{atn}
k\cdot|\Lambda|\cdot\nu(B_{\frac{1}{4k+4},x,\Lambda})\leq 1-\frac{1}{4k^2+4k+1}
\end{equation}
Theorem 2.1 from \cite{DQ}, implies that the system $(\Sigma_{k+1},\mathfrak{B}_{k+1},\nu, S_{k+1})$ is not AT. Then (\ref{atn}) and Theorem \ref{t1} implies that the system is not AT($2$). Applying successively the same Theorem \ref{t1}, it results that the system $(\Sigma_{k+1},\mathfrak{B}_{k+1},\nu, S_{k+1})$ is not AT($k$). Using now Proposition \ref{xx}, it follows that that $T$ is not AT($k$). Since $k\geq 2$ was arbitrary chosen, the proposition is proved.
\end{proof}

\subsection*{Acknowledgement.} This work was supported by a grant of the Romanian Ministry of Education, CNCS - UEFISCDI, project number PN-II-RU-PD-2012-3-0533.

%Then the expectation $E(S)=\int_{\Sigma_k}S d\nu$ of $S$ is zero and $|S|^2<a_{1}^2+a_2^2+\cdots+a_{k+1}^2$.

\end{document}